\documentclass[11pt,reqno]{amsart} 

\usepackage{amssymb,latexsym}
\usepackage[draft=true]{hyperref}
\usepackage{cite} 

\usepackage[height=190mm,width=130mm]{geometry} 

\theoremstyle{plain}
\newtheorem{theorem}{Theorem}
\newtheorem{lemma}{Lemma}

\theoremstyle{definition}
\newtheorem{definition}{Definition}

\theoremstyle{remark}

\numberwithin{equation}{section} 

\begin{document}
\title{Inverse problem for singular diffusion operator } 

\author{ Abdullah ERG\"{U}N}
\address{Cumhuriyet University\\ Vocational School of Sivas\\ Sivas.\\ 58140
  \\ Turkey}

\email{aergun@cumhuriyet.edu.tr}

\begin{abstract}
In this study, singular diffusion operator with jump conditions is
considered. Integral representations have been derived for solutions that
satisfy boundary conditions and jump conditions. Some properties of
eigenvalues and eigenfunctions are investigated. Asymtotic representation of eigenvalues and eigenfunction have been obtained. Reconstruction of the singular diffusion operator  have been shown by the Weyl function.
\end{abstract}




\maketitle

\section{ Introduction.}

 Let's define the following boundary value problem which will be denoted by $L$ in the sequel all the paper
\begin{equation} \label{1)} 
l\left(y\right):=-y''+\left[2\lambda p\left(x\right)+q\left(x\right)\right]y=\lambda ^{2} \delta \left(x\right)y,\, \, x\in \left[0,\pi \right]/\left\{p_{1} ,p_{2} \right\} 
\end{equation} 
with the boundary conditions
\begin{equation} \label{2)} 
y'\left(0\right)=0,y\left(\pi \right)=0 
\end{equation} 
and the jump conditions
\begin{equation} \label{3)} 
y\left(p_{1} +0\right)=\alpha _{1} y\left(p_{1} -0\right) 
\end{equation} 
\begin{equation} \label{4)} 
y'\left(p_{1} +0\right)=\beta _{1} y'\left(p_{1} -0\right)+i\lambda \gamma _{1} y\left(p_{1} -0\right) 
\end{equation} 
\begin{equation} \label{5)} 
y\left(p_{2} +0\right)=\alpha _{2} y\left(p_{2} -0\right) 
\end{equation} 
\begin{equation} \label{6)} 
y'\left(p_{2} +0\right)=\beta _{2} y'\left(p_{2} -0\right)+i\lambda \gamma _{2} y\left(p_{2} -0\right) 
\end{equation} 
where $\lambda $ is a spectral parameter, $q(x)\in L_{2} \left[0,\pi \right]$, $p(x)\in W_{2}^{1} \left[0,\pi \right]$,\\ $p_{1} ,p_{2} \in \left(0,\pi \right)$, $p_{1} <p_{2} $, $\left|\alpha _{1} -1\right|^{2} +\gamma _{1} ^{2} \ne 0$, $\left|\alpha _{2} -1\right|^{2} +\gamma _{2} \ne 0$, $\left(\beta _{i} =\frac{1}{\alpha _{i} } \left(i=1,2\right)\right)$ and

\noindent $\delta \left(x\right)=\left\{%
\begin{array}{l}
{1\, \, \, \, ,\, \, \, \, x\in \left(0,p_{1} \right)} \\
{\alpha ^{2} ,\, \, \, \, x\in \left(p_{1} ,p_{2} \right)} \\
{\beta ^{2} ,\, \, \, \, x\in \left(p_{2} ,\pi \right)}%
\end{array}%
\right. $ to be $\alpha >0\, \, ,\, \alpha \ne 1$,$\beta >0\, \, ,\, \beta
\ne 1$ real numbers.

\noindent Direct and inverse problems are important in mathematics, physics
and engineering. The inverse problem is called the reconstruction of the
operator whose spectral characteristics are given in sequences. For example;
to learn the distribution of density in the nonhomogeneous arc according to
the wave lengths in mechanics and finding the field potentials according to
scattering data in the quantum physics are examples of inverse problems. The
first study on inverse problems for differential equations was made by
Ambartsumyan $\left[25\right]$. A significant study in the spectral theory
of the singular differential operators was carried out by Levitan in $\left[4%
\right]$. An important method in the solution of inverse problems is the
transformation operators. 
Guseinov\cite{Guseinov-2} studied the regular differential equation and the direct
spectral problem of the operator under certain initial conditions. In recent
years, Weyl function has frequently been used to solve inverse problems. The
Weyl function was introduced by H. Weyl(1910) in the literature. Many
studies have been made on direct or inverse problems \cite{2,4,Amirov,Anderson,Borg,Carlson,Ergun-1,Gasymov,Guseinov-1,Guseinov-2,Hald,Jdanovich,Krein,Levin,Levitan-1,Levitan-2,Marchenko,Nabiev,Yang-1,Yurko,Yurko-1,koyunbakan,levitan,Ergun-2,Ergun-3,Gala,Gala-1,Maris,ragusa} . The
solution of discontinuous boundary value problem can be given as an example
of concrete problem of mathematical physics. Boundary value problems with
discontinuous coefficients are important for applied mathematics and applied
sciences.

\noindent H. Koyunbakan, E. S. Panakhov\cite{koyunbakan} proved that the potential function
can be determined on $\left[\frac{\pi }{2} ,\pi \right]$while it is known on
$\left[0,\frac{\pi }{2} \right]$ by single spectrum in$\left[11\right]$. C.
F. Yang \cite{Yang-1}  showed that  can be
determined uniquely  diffusion operator from nodal data.

\section{Preliminaries.}

\noindent Let $\phi \left(x,\lambda \right)$, $\psi \left(x,\lambda \right)$
be solutions of $\left(1.1\right)$ respectively under the boundary conditions
$$
\phi \left(0,\lambda \right)=1,\phi ^{\prime }\left(0,\lambda \right)=0
$$
$$
\psi \left(\pi ,\lambda \right)=0,\psi ^{\prime }\left(\pi ,\lambda \right)=1
$$
and discontinuity conditions $\left(1.3\right)-\left(1.6\right)$, where \textbf{%
	$Q\left(t\right)=2\lambda p\left(t\right)+q\left(t\right).$}

\noindent It is obvious that the function $\phi \left(x,\lambda \right)$ is similar to \cite{Ergun-2}
satisfies the following integral equations

if $0\leq x<p_{1}$ ;
\begin{equation} \label{2)} 
\phi \left( x,\lambda \right) =e^{i\lambda x}+\frac{1}{\lambda }%
\int_{0}^{x}\sin \lambda \left( x-t\right) Q\left( t\right) y\left(
t,\lambda \right) dt  
\end{equation}%
if $p_{1}<x<p_{2}$;
\begin{equation}
\begin{array}{l}
{\phi \left( x,\lambda \right) =\beta _{1}^{+}e^{i\lambda \varsigma
		^{+}\left( x\right) }+\beta _{1}^{-}e^{i\lambda \varsigma ^{-}\left(
		x\right) }+\frac{\gamma _{1}}{2\alpha }e^{i\lambda \varsigma ^{+}\left(
		x\right) }-\frac{\gamma _{1}}{2\alpha }e^{i\lambda \varsigma ^{-}\left(
		x\right) }} \\
{+\beta _{1}^{+}\int_{0}^{p_{1}}\frac{\sin \lambda \left( \varsigma
		^{+}\left( x\right) -t\right) }{\lambda }J\left( t\right) y\left( t,\lambda
	\right) dt+\beta _{1}^{-}\int_{0}^{p_{1}}\frac{\sin \lambda \left( \varsigma
		^{-}\left( x\right) -t\right) }{\lambda }J\left( t\right) y\left( t,\lambda
	\right) dt} \\
{-i\frac{\gamma _{1}}{2\alpha }\int_{0}^{p_{1}}\frac{\cos \lambda \left(
		\varsigma ^{+}\left( x\right) -t\right) }{\lambda }J\left( t\right) y\left(
	t,\lambda \right) dt++i\frac{\gamma _{1}}{2\alpha }\int_{0}^{p_{1}}\frac{%
		\cos \lambda \left( \varsigma ^{-}\left( x\right) -t\right) }{\lambda }%
	J\left( t\right) y\left( t,\lambda \right) dt} \\
{+\int_{p_{1}}^{x}\frac{\sin \lambda \left( x-t\right) }{\lambda }J\left(
	t\right) y\left( t,\lambda \right) dt}%
\end{array}
\label{8}
\end{equation}

if $p_{2}<x\leq \pi $ ;

\begin{equation}
\begin{array}{l}
{\phi \left(x,\lambda \right)=\xi ^{+} e^{i\lambda b^{+} \left(x\right)} +\xi ^{-} e^{i\lambda b^{-} \left(x\right)} +\vartheta ^{+} e^{i\lambda s^{+} \left(x\right)} +\vartheta ^{-} e^{i\lambda s^{-} \left(x\right)} } \\ {+\left(\beta _{1} ^{+} \beta _{2} ^{+} +\frac{\gamma _{1} \gamma _{2} }{4\alpha \beta } \right)\int _{0}^{p_{1} }\frac{\sin \lambda \left(b^{+} \left(x\right)-t\right)}{\lambda }  J\left(t\right)y\left(t,\lambda \right)dt} \\ {+\left(\beta _{1} ^{+} \beta _{2} ^{-} -\frac{\gamma _{1} \gamma _{2} }{4\alpha \beta } \right)\int _{0}^{p_{1} }\frac{\sin \lambda \left(s^{+} \left(x\right)-t\right)}{\lambda }  J\left(t\right)y\left(t,\lambda \right)dt} \\ {+\left(\beta _{1} ^{-} \beta _{2} ^{-} -\frac{\gamma _{1} \gamma _{2} }{4\alpha \beta } \right)\int _{0}^{p_{1} }\frac{\sin \lambda \left(b^{-} \left(x\right)-t\right)}{\lambda }  J\left(t\right)y\left(t,\lambda \right)dt} \\ {+\left(\beta _{1} ^{-} \beta _{2} ^{+} +\frac{\gamma _{1} \gamma _{2} }{4\alpha \beta } \right)\int _{p_{1} }^{p_{2} }\frac{\sin \lambda \left(s^{-} \left(x\right)-t\right)}{\lambda }  J\left(t\right)y\left(t,\lambda \right)dt} \\ {-i\left(\frac{\gamma _{1} \beta _{2} ^{+} }{2\alpha } +\frac{\gamma _{2} \beta _{1} ^{+} }{2\beta } \right)\int _{0}^{p_{1} }\frac{\cos \lambda \left(b^{+} \left(x\right)-t\right)}{\lambda }  J\left(t\right)y\left(t,\lambda \right)dt} \\ {-i\left(\frac{\gamma _{1} \beta _{2} ^{-} }{2\alpha } -\frac{\gamma _{2} \beta _{1} ^{+} }{2\beta } \right)\int _{0}^{p_{1} }\frac{\cos \lambda \left(s^{+} \left(x\right)-t\right)}{\lambda }  J\left(t\right)y\left(t,\lambda \right)dt} \\ {+i\left(\frac{\gamma _{1} \beta _{2} ^{-} }{2\alpha } -\frac{\gamma _{2} \beta _{1} ^{-} }{2\beta } \right)\int _{0}^{p_{1} }\frac{\cos \lambda \left(b^{-} \left(x\right)-t\right)}{\lambda }  J\left(t\right)y\left(t,\lambda \right)dt} \\ {+i\left(\frac{\gamma _{1} \beta _{2} ^{+} }{2\alpha } +\frac{\gamma _{2} \beta _{1} ^{-} }{2\beta } \right)\int _{p_{1} }^{p_{2} }\frac{\cos \lambda \left(s^{-} \left(x\right)-t\right)}{\lambda }  J\left(t\right)y\left(t,\lambda \right)dt} \\ {+\beta _{2} ^{+} \int _{p_{1} }^{p_{2} }\frac{\sin \lambda \left(\beta x-\beta p_{2} +\alpha p_{2} -\alpha t\right)}{\lambda }  J\left(t\right)y\left(t,\lambda \right)dt} \\ {-\beta _{2} ^{-} \int _{p_{1} }^{p_{2} }\frac{\sin \lambda \left(\beta x-\beta p_{2} -\alpha p_{2} +\alpha t\right)}{\lambda }  J\left(t\right)y\left(t,\lambda \right)dt} \\ {-i\frac{\gamma _{2} }{2\beta } \int _{p_{1} }^{p_{2} }\frac{\cos \lambda \left(\beta x-\beta p_{2} +\alpha p_{2} -\alpha t\right)}{\lambda }  J\left(t\right)y\left(t,\lambda \right)dt} \\ {+i\frac{\gamma _{2} }{2\beta } \int _{p_{1} }^{p_{2} }\frac{\cos \lambda \left(\beta x-\beta p_{2} -\alpha p_{2} +\alpha t\right)}{\lambda }  J\left(t\right)y\left(t,\lambda \right)dt} \\ {+\int _{p_{2} }^{x}\frac{\sin \lambda \left(x-t\right)}{\lambda }  J\left(t\right)y\left(t,\lambda \right)dt\, \, \, \, \, \, \, \, \, \, \, \, \, \, \, \, \, \, \, \, \, \, \, \, \, \, \, \, \, \, \, \, \, \, \, \, \, \, \, \, \, \, \, \, \, \, \, \, \, \, \, \, \, \, \, \, \, \, \, \, \, \, \, \, \, \, \, \, \, \, \, \, \, \, \, \, \, \, \, \, \, \, \, \, \, \, \, \, \, \, \, \, \, \, \, \, \, \, \, \, \, \, \, \, \, \, \, \, \, \, \, \, \, \, \, \, } 
\end{array}
\label{9}
\end{equation}%
and it is obvious that the function $\psi \left( x,\lambda \right) $
satisfies the following integral equations;

\noindent

\noindent if $p_{2}<x\leq \pi $,
\begin{equation}
\psi \left( x,\lambda \right) =\frac{\sin \lambda \beta \left( x-\pi \right)
}{\lambda \beta }+\int_{x}^{\pi }\frac{\sin \lambda \beta \left( x-t\right)
}{\lambda \beta }Q\left( t\right) y\left( t,\lambda \right) dt  \label{10}
\end{equation}%
\ if $p_{1}<x<p_{2}$;
\begin{equation}
\begin{array}{l}
{\psi \left(x,\lambda \right)=} \\ {\left(\frac{\alpha \beta _{2} -\gamma _{2} }{2\alpha \beta _{2} \lambda \alpha _{2} \beta } -\frac{1}{2\alpha \beta _{2} \lambda } \right)e^{-i\lambda \left(\beta \left(p_{2} -\pi \right)+\alpha \left(p_{2} -x\right)\right)} } \\ {+\left(\frac{\alpha \beta _{2} +\gamma _{2} }{2\alpha \beta _{2} \lambda \alpha _{2} \beta } +\frac{1}{2\alpha \beta _{2} \lambda } \right)e^{-i\lambda \left(\beta \left(p_{2} -\pi \right)-\alpha \left(p_{2} -x\right)\right)} } \\ {-\left(\frac{\alpha \beta _{2} -\gamma _{2} }{2\alpha \beta _{2} } -\frac{1}{2} \right)\int _{p_{1} }^{p_{2} }\frac{\sin \lambda \left(x-p_{2} +\alpha t-\alpha p_{2} \right)}{\lambda \alpha }  Q\left(t\right)y\left(t,\lambda \right)dt} \\ {+\left(\frac{\alpha \beta _{2} -\gamma _{2} }{2\alpha \beta _{2} } +\frac{1}{2} \right)\int _{p_{1} }^{p_{2} }\frac{\sin \lambda \left(x-p_{2} -\alpha t+\alpha p_{2} \right)}{\lambda \alpha }  Q\left(t\right)y\left(t,\lambda \right)dt} \\ {+\frac{1}{2} \left(\frac{\alpha \beta _{2} -\gamma _{2} }{\alpha \beta _{2} \alpha _{2} \beta } -\frac{1}{\alpha \beta _{2} } \right)\int _{p_{2} }^{\pi }\frac{\sin \lambda \left(x-p_{2} +\beta \left(t-p_{2} \right)\right)}{\lambda \beta }  Q\left(t\right)y\left(t,\lambda \right)dt} \\ {-\frac{1}{2} \left(\frac{\alpha \beta _{2} -\gamma _{2} }{\alpha \beta _{2} \alpha _{2} \beta } -\frac{1}{\alpha \beta _{2} } \right)\int _{p_{2} }^{\pi }\frac{\sin \lambda \left(x-p_{2} -\beta \left(t-p_{2} \right)\right)}{\lambda \beta }  Q\left(t\right)y\left(t,\lambda \right)dt} \\ {+\frac{\gamma _{2} }{2\alpha \beta _{2} \lambda } \int _{p_{1} }^{p_{2} }\frac{\cos \lambda \left(x-p_{2} +\alpha t-\alpha p_{2} \right)}{\lambda \alpha }  Q\left(t\right)y\left(t,\lambda \right)dt} \\ {-\frac{\gamma _{2} }{2\alpha \beta _{2} \lambda } \int _{p_{1} }^{p_{2} }\frac{\cos \lambda \left(x-p_{2} -\alpha t+\alpha p_{2} \right)}{\lambda \alpha }  Q\left(t\right)y\left(t,\lambda \right)dt+\int _{p_{1} }^{x}\frac{\sin \lambda \alpha \left(x-t\right)}{\lambda \alpha }  Q\left(t\right)y\left(t,\lambda \right)dt}
\end{array}
\label{11}
\end{equation}%
\ if $0\leq x<p_{1}$;
$$
\begin{array}{l}
{\psi \left(x,\lambda \right)=} \\ {\left(\xi ^{+} +\frac{\alpha }{2\beta _{1} } \right)\eta ^{-} e^{-i\lambda \left(b^{-} \left(\pi \right)+x\right)} +\left(\xi ^{-} -\frac{\alpha }{2\beta _{1} } \right)\eta ^{+} e^{-i\lambda \left(b^{+} \left(\pi \right)+x\right)} } \\ {+\left(\xi ^{-} +\frac{\alpha }{2\beta _{1} } \right)\eta ^{-} e^{-i\lambda \left(s^{+} \left(\pi \right)+x\right)} +\left(\xi ^{-} -\frac{\alpha }{2\beta _{1} } \right)\eta ^{+} e^{-i\lambda \left(s^{-} \left(\pi \right)+x\right)} } \\ {+\left(\frac{1}{2\alpha _{1} } -\frac{\mu ^{+} }{4\beta _{1} } \right)\int _{a_{2} }^{\pi }\frac{\sin \lambda \left(x-p_{2} -\beta t+\beta p_{2} \right)}{\lambda }  Q\left(t\right)y\left(t,\lambda \right)dt} \\ {-\left(\frac{1}{2\alpha _{1} } +\frac{\mu ^{+} }{4\beta _{1} } \right)\int _{p_{2} }^{\pi }\frac{\sin \lambda \left(x-2p_{1} +p_{2} +\beta t-\beta p_{2} \right)}{\lambda }  Q\left(t\right)y\left(t,\lambda \right)dt} \\ {+\left(\frac{1}{2\alpha _{1} } +\frac{\mu ^{-} }{4\beta _{1} } \right)\int _{p_{2} }^{\pi }\frac{\sin \lambda \left(x-p_{2} -\beta t+\beta p_{2} \right)}{\lambda }  Q\left(t\right)y\left(t,\lambda \right)dt} \\ {-\left(\frac{1}{2\alpha _{1} } -\frac{\mu ^{-} }{4\beta _{1} } \right)\int _{p_{2} }^{\pi }\frac{\sin \lambda \left(x-2p_{1} +p_{2} -\beta t+\beta p_{2} \right)}{\lambda }  Q\left(t\right)y\left(t,\lambda \right)dt} \\ {+\frac{i\gamma _{1} }{2\alpha _{1} \beta _{1} } \int _{p_{2} }^{\pi }\frac{\cos \lambda \left(x-p_{2} +\beta t-\beta p_{2} \right)}{\lambda }  Q\left(t\right)y\left(t,\lambda \right)dt} \\ {-\frac{i\gamma _{1} }{2\alpha _{1} \beta _{1} } \int _{p_{2} }^{\pi }\frac{\cos \lambda \left(x-2p_{1} +p_{2} -\beta t+\beta p_{2} \right)}{\lambda }  Q\left(t\right)y\left(t,\lambda \right)dt} \\ {-\frac{i\gamma _{1} }{2\alpha _{1} \beta _{1} } \int _{p_{2} }^{\pi }\frac{\cos \lambda \left(x-p_{2} -\beta t+\beta p_{2} \right)}{\lambda }  Q\left(t\right)y\left(t,\lambda \right)dt} \\ {+\frac{i\gamma _{1} }{2\alpha _{1} \beta _{1} } \int _{p_{2} }^{\pi }\frac{\cos \lambda \left(x-2p_{1} +p_{2} +\beta t-\beta p_{2} \right)}{\lambda }  Q\left(t\right)y\left(t,\lambda \right)dt} \\ {+A\int _{p_{1} }^{p_{2} }\frac{\sin \lambda \left(x-p_{2} +\alpha t-\alpha p_{2} \right)}{\lambda \alpha }  Q\left(t\right)y\left(t,\lambda \right)dt+A\int _{p_{1} }^{p_{2} }\frac{\sin \lambda \left(x-2p_{1} +p_{2} -\alpha t+\alpha p_{2} \right)}{\lambda \alpha }  Q\left(t\right)y\left(t,\lambda \right)dt} \\ {+B\int _{p_{1} }^{p_{2} }\frac{\cos \lambda \left(x-p_{2} +\alpha t-\alpha p_{2} \right)}{\lambda \alpha }  Q\left(t\right)y\left(t,\lambda \right)dt+B\int _{p_{1} }^{p_{2} }\frac{\cos \lambda \left(x-2p_{1} +p_{2} -\alpha t+\alpha p_{2} \right)}{\lambda \alpha }  Q\left(t\right)y\left(t,\lambda \right)dt}
\end{array}%
$$
\begin{equation}
\begin{array}{l}
{+C\int_{p_{1}}^{p_{2}}\frac{\sin \lambda \left( x-p_{2}-\alpha t+\alpha
		p_{2}\right) }{\lambda \alpha }Q\left( t\right) y\left( t,\lambda \right)
	dt+C\int_{p_{1}}^{p_{2}}\frac{\sin \lambda \left( x-2p_{1}+p_{2}+\alpha
		t-\alpha p_{2}\right) }{\lambda \alpha }Q\left( t\right) y\left( t,\lambda
	\right) dt} \\
{+D\int_{p_{1}}^{p_{2}}\frac{\cos \lambda \left( x-p_{2}-\alpha t+\alpha
		p_{2}\right) }{\lambda \alpha }Q\left( t\right) y\left( t,\lambda \right)
	dt+D\int_{p_{1}}^{p_{2}}\frac{\cos \lambda \left( x-2p_{1}+p_{2}+\alpha
		t-\alpha p_{2}\right) }{\lambda \alpha }Q\left( t\right) y\left( t,\lambda
	\right) dt} \\
{+\int_{0}^{x}\frac{\sin \lambda \left( x-t\right) }{\lambda }Q\left(
	t\right) y\left( t,\lambda \right) dt}%
\end{array}
\label{12}
\end{equation}%
where $\varsigma ^{\pm } \left(x\right)=\pm \alpha x\mp \alpha p_{1} +p_{1} \, \, ,\, \, \beta _{1} ^{\pm } =\frac{1}{2} \left(\alpha _{1} \pm \frac{\beta _{1} }{\alpha } \right)\, \, ,$
\[b^{\pm } \left(x\right)=\beta x-\beta p_{2} +\mu ^{\pm } \left(p_{2} \right)\, \, ,\, \, s^{\pm } \left(x\right)=-\beta x+\beta p_{2} +\mu ^{\pm } \left(p_{2} \right)\, \, ,\] 
\[\beta _{2} ^{\mp } =\frac{1}{2} \left(\alpha _{2} \mp \frac{\alpha \beta _{2} }{\beta } \right)\, \, ,\, \, \xi ^{\mp } =\frac{1}{2} \left(\beta _{1} ^{\mp } \mp \frac{\gamma _{1} }{2\alpha } \right)\left(\alpha _{2} \mp \frac{\alpha \beta _{2} }{\beta } +\frac{\gamma _{2} }{\beta } \right)\, \, ,\] 
\[\vartheta ^{\mp } =\frac{1}{2} \left(\beta _{1} ^{\mp } \mp \frac{\gamma _{1} }{2\alpha } \right)\left(\alpha _{2} \pm \frac{\alpha \beta _{2} }{\beta } -\frac{\gamma _{2} }{\beta } \right)\, \, ,\, \, \mu ^{\pm } =\left(\frac{\alpha \beta _{2} \pm \gamma _{2} }{2\alpha \beta _{2} \lambda \alpha _{2} \beta } \pm \frac{1}{2\alpha \beta _{2} \lambda } \right)\, \, ,\] 
\[\begin{array}{l} {A=\left[\left(\frac{i\gamma _{1} \gamma _{2} }{4\lambda \alpha \alpha _{1} \beta _{1} \beta _{2} } +\left(\frac{-1}{2\alpha _{1} } -\frac{1}{4\beta _{1} } \right)\left(\frac{\alpha \beta _{2} -\gamma _{2} }{2\alpha \beta _{2} } -\frac{1}{2} \right)\right)\right]\, ,} \\ {B=\left[\frac{-i\gamma _{1} }{2\alpha _{1} \beta _{1} } \left(\frac{\alpha \beta _{2} -\gamma _{2} }{2\alpha \beta _{2} } -\frac{1}{2} \right)+\frac{1}{2\alpha _{1} } \frac{\gamma _{2} }{2\alpha \beta _{2} \lambda } \right],} \\ {C=\left[\left(\frac{1}{2\alpha _{1} } +\frac{1}{2\beta _{1} } \right)\left(\frac{\alpha \beta _{2} -\gamma _{2} }{2\alpha \beta _{2} } +\frac{1}{2} \right)+\frac{i\gamma _{1} \gamma _{2} }{4\lambda \alpha \alpha _{1} \beta _{1} \beta _{2} } \right]\, ,} \\ {\, D=\left[\frac{i\gamma _{1} }{2\alpha _{1} \beta _{1} } \left(\frac{\alpha \beta _{2} -\gamma _{2} }{2\alpha \beta _{2} } +\frac{1}{2} \right)-\frac{\gamma _{2} \left(1-\alpha ^{2} \right)}{4\alpha _{1} \alpha \beta _{2} \lambda } \right].} \end{array}\]

\begin{theorem}
	If  $p\left(x\right)\in W_{2}^{1} \left(0,\pi \right)$ and $q\left(x\right)\in L_{2} \left(0,\pi \right)$;$\; y_{\upsilon } \left(x,\lambda \right)$ be solutions of $\left(1.1\right)$, that satisfies conditions $\left(1.2\right)-\left(1.6\right)$, has the form
	\[y_{\upsilon } \left(x,\lambda \right)=y_{0\upsilon } \left(x,\lambda \right)+\int _{-x}^{x}K_{\upsilon }  \left(x,t\right)e^{i\lambda t} dt \left(\upsilon =\overline{1,3}\right)\] 
	where
	\[y_{0\upsilon } \left(x,\lambda \right)=\left\{\begin{array}{l} {R_{0} \left(x\right)e^{i\lambda x} \, \, \, \, \, \, \, \, \, \, \, \, \, \, \, \, \, \, \, \, \, \, \, \, \, \, \, \, \, \, \, \, \, \, \, \, \, \, \, \, \, \, \, \, \, \, \, \, \, \, \, \, \, \, \, \, \, \, \, \, \, \, \, \, \, \, \, \, \, \, \, \, \, \, \, \, \, \, \, \, \, \, \, \, \, \, \, \, \, \, \, \, \, \, \, \, \, \, ;0\le x<p_{1} } \\ {R_{1} \left(x\right)e^{i\lambda \varsigma ^{+} \left(x\right)} +R_{2} \left(x\right)e^{i\lambda \varsigma ^{-} \left(x\right)} \, \, \, \, \, \, \, \, \, \, \, \, \, \, \, \, \, \, \, \, \, \, \, \, \, \, \, \, \, \, \, \, \, \, \, \, \, \, \, \, \, \, \, \, \, \, \, \, \, \, \, \, \, \, \, \, \, \, \, \, \, \, \, \, ;p_{1} <x<p_{2} } \\ {R_{3} \left(x\right)e^{i\lambda b^{+} \left(x\right)} +R_{4} \left(x\right)e^{i\lambda b^{-} \left(x\right)} +R_{5} \left(x\right)e^{i\lambda s^{+} \left(x\right)} +R_{6} \left(x\right)e^{i\lambda s^{-} \left(x\right)} \, \, \, \, ;p_{2} <x\le \pi } \end{array}\right. \] 
	\[R_{0} \left(x\right)=e^{-i\int _{0}^{x}p\left(x\right)dx } \, \, ,\, \, R_{1} \left(x\right)=\left(\beta _{1} ^{+} +\frac{\gamma _{1} }{2\alpha } \right)R_{0} \left(p_{1} \right)e^{-\frac{i}{\alpha } \int _{p_{1} }^{x}p\left(t\right) dt} \, \, ,\] 
	\[R_{2} \left(x\right)=\left(\beta _{1} ^{-} -\frac{\gamma _{1} }{2\alpha } \right)R_{0} \left(p_{1} \right)e^{\frac{i}{\alpha } \int _{p_{1} }^{x}p\left(t\right) dt} \, \, ,\, \, R_{3} \left(x\right)=\left(\beta _{2} ^{+} +\frac{\gamma _{2} }{2\beta } \right)R_{1} \left(p_{2} \right)e^{-\frac{i}{\beta } \int _{p_{2} }^{x}p\left(t\right) dt} \, \, ,\] 
	\[R_{4} \left(x\right)=\left(\beta _{2} ^{-} +\frac{\gamma _{2} }{2\beta } \right)R_{2} \left(p_{2} \right)e^{-\frac{i}{\beta } \int _{p_{2} }^{x}p\left(t\right) dt} \, \, ,\, \, R_{5} \left(x\right)=\left(\beta _{2} ^{-} -\frac{\gamma _{2} }{2\beta } \right)R_{1} \left(p_{2} \right)e^{\frac{i}{\beta } \int _{p_{2} }^{x}p\left(t\right) dt} \, \, ,\] 
	\[R_{6} \left(x\right)=\left(\beta _{2} ^{+} -\frac{\gamma _{2} }{2\beta } \right)R_{2} \left(p_{2} \right)e^{\frac{i}{\beta } \int _{p_{2} }^{x}p\left(t\right) dt} \] 
	and $\varpi \left(x\right)=\int _{0}^{x}\left(2\left|p\left(t\right)\right|+\left(x-t\right)\left|q\left(t\right)\right|\right) dt$ and the functions $K_{\upsilon } \left(x,t\right)$ satisfies the inequality
	\[\int _{-x}^{x}\left|K_{\upsilon } \left(x,\lambda \right)\right| dt\le e^{c_{\upsilon } \varpi \left(x\right)} -1\] 
	with 
	\[c_{1} =1 ,c_{2} =\left(\beta _{1}^{+} +\left|\beta _{1}^{-} \right|+\frac{\gamma _{1} }{\alpha } +\frac{2}{\alpha } \right) , c_{3} =\left(\alpha _{2} \left(\beta _{1}^{+} +\left|\beta _{1}^{-} \right|\right)+\frac{1}{\alpha } \left(\beta _{2}^{+} +\left|\beta _{2}^{-} \right|\right)+\frac{\beta ^{+} }{\beta } +\frac{\gamma _{2} }{\beta } \right)\] 
	where  $\varsigma ^{\pm } \left(x\right)=\pm \alpha x\mp \alpha p_{1} +p_{1} $, $\beta _{1} ^{\pm } =\frac{1}{2} \left(\alpha _{1} \pm \frac{\beta _{1} }{\alpha } \right)$, $b^{\pm } \left(x\right)=\beta x-\beta p_{2} +\varsigma ^{\pm } \left(p_{2} \right)$, $s^{\pm } \left(x\right)=-\beta x+\beta p_{2} +\varsigma ^{\pm } \left(p_{2} \right)$, $\beta _{2} ^{\mp } =\frac{1}{2} \left(\alpha _{2} \mp \frac{\alpha \beta _{2} }{\beta } \right)$, $\xi ^{\mp } =\frac{1}{2} \left(\beta _{1} ^{\mp } \mp \frac{\gamma _{1} }{2\alpha } \right)\left(\alpha _{2} \mp \frac{\alpha \beta _{2} }{\beta } +\frac{\gamma _{2} }{\beta } \right)$, $\vartheta ^{\mp } =\frac{1}{2} \left(\beta _{1} ^{\mp } \mp \frac{\gamma _{1} }{2\alpha } \right)\left(\alpha _{2} \pm \frac{\alpha \beta _{2} }{\beta } -\frac{\gamma _{2} }{\beta } \right)$, $\beta ^{\pm } =\frac{1}{2} \left(1\pm \frac{1}{\beta } \right)$.
\end{theorem}
The proof is done as in \cite{Ergun-2}.

\begin{theorem}
	Let $p\left( x\right) \in W_{2}^{1}\left( 0,\pi \right) $ and $q\left(
	x\right) \in L_{2}\left( 0,\pi \right) $. The functions $A\left( x,t\right) $%
	, $B\left( x,t\right) $, whose first order partial derivatives, are summable
	on $\left[ 0,\pi \right] $, for each $x\in \left[ 0,\pi \right] $ such that
	representation
	$$
	\varphi \left( x,\lambda \right) =\varphi _{0}\left( x,\lambda \right)
	+\int_{0}^{x}A\left( x,t\right) \cos \lambda tdt+\int_{0}^{x}B\left(
	x,t\right) \sin \lambda tdt
	$$
	is satisfied.\\
\ If $p_{1}<x<p_{2}$;
\end{theorem}

\begin{equation}
\begin{array}{l}
{\varphi \left(x,\lambda \right)=\left(\beta _{1} ^{+} +\frac{\gamma _{1} }{2\alpha } \right)R_{0} \left(p_{1} \right)\cos \left[\lambda \varsigma ^{+} \left(x\right)-\frac{1}{\alpha } \int _{p_{1} }^{x}p\left(t\right)dt \right]} \\ {+\left(\beta _{1} ^{-} -\frac{\gamma _{1} }{2\alpha } \right)R_{0} \left(p_{1} \right)\cos \left[\lambda \varsigma ^{-} \left(x\right)+\frac{1}{\alpha } \int _{p_{1} }^{x}p\left(t\right)dt \right]} \\ {+\int _{0}^{\varsigma ^{+} \left(x\right)}A \left(x,t\right)\cos \lambda tdt+\int _{0}^{\varsigma ^{+} \left(x\right)}B \left(x,t\right)\sin \lambda tdt}
\end{array}
\label{13}
\end{equation}%
\\
where $\beta _{1}^{\pm }=\frac{1}{2}\left( \alpha _{1}\pm \frac{\beta _{1}}{%
	\alpha }\right). $

\noindent If $p_{2}<x\leq \pi $,
\begin{equation}
\begin{array}{l}
{\varphi \left(x,\lambda \right)=\left(\beta _{2} ^{+} +\frac{\gamma _{2} }{2\beta } \right)R_{1} \left(p_{2} \right)\cos \left[\lambda b^{+} \left(x\right)-\frac{1}{\beta } \int _{p_{2} }^{x}p\left(t\right)dt \right]} \\ {+\left(\beta _{2} ^{-} +\frac{\gamma _{2} }{2\beta } \right)R_{2} \left(p_{2} \right)\cos \left[\lambda b^{-} \left(x\right)-\frac{1}{\beta } \int _{p_{2} }^{x}p\left(t\right)dt \right]} \\ {+\left(\beta _{2} ^{-} -\frac{\gamma _{2} }{2\beta } \right)R_{1} \left(p_{2} \right)\cos \left[\lambda s^{+} \left(x\right)+\frac{1}{\beta } \int _{p_{2} }^{x}p\left(t\right)dt \right]} \\ {+\left(\beta _{2} ^{+} -\frac{\gamma _{2} }{2\beta } \right)R_{2} \left(p_{2} \right)\cos \left[\lambda s^{-} \left(x\right)+\frac{1}{\beta } \int _{p_{2} }^{x}p\left(t\right)dt \right]} \\ {+\int _{p_{2} }^{x}A\left(x,t\right)\cos \lambda tdt +\int _{p_{2} }^{x}B\left(x,t\right)\sin \lambda tdt }
\end{array}
\label{14}
\end{equation}%
where $\beta _{2}^{\mp }=\frac{1}{2}\left( \alpha _{2}\mp \frac{\alpha \beta
	_{2}}{\beta }\right). $

\noindent Moreover, the equations

\noindent
\begin{equation}
\begin{array}{l} {A\left(x,\varsigma ^{+} \left(x\right)\right)\cos \frac{\beta \left(x\right)}{\alpha } +B\left(x,\varsigma ^{+} \left(x\right)\right)\sin \frac{\beta \left(x\right)}{\alpha } } \\ {=\left(\beta _{1} ^{+} +\frac{\gamma _{1} }{2\alpha } \right)\frac{R_{0} \left(p_{1} \right)}{2\alpha } \int _{0}^{x}\left(q\left(t\right)+\frac{p^{2} \left(t\right)}{\alpha ^{2} } \right) dt\, } \end{array}  \label{15}
\end{equation}%
\begin{equation}
\begin{array}{l} {A\left(x,\varsigma ^{+} \left(x\right)\right)\sin \frac{\beta \left(x\right)}{\alpha } -B\left(x,\varsigma ^{+} \left(x\right)\right)\cos \frac{\beta \left(x\right)}{\alpha } } \\ {=\left(\beta _{1} ^{+} +\frac{\gamma _{1} }{2\alpha } \right)\frac{R_{0} \left(p_{1} \right)}{2\alpha ^{2} } \left(p\left(x\right)-p\left(0\right)\right)\, } \end{array}\label{16}
\end{equation}

\begin{equation}
\begin{array}{l}
{A\left( x,\varsigma ^{-}\left( x\right) +0\right) -A\left( x,\varsigma
	^{-}\left( x\right) -0\right) =} \\
{\left( \beta _{1}^{-}-\frac{\gamma _{1}}{2\alpha }\right) \frac{R_{0}\left(
		p_{1}\right) }{2\alpha ^{2}}\sin \frac{\beta \left( x\right) }{\alpha }%
	\left( p\left( x\right) -p\left( 0\right) \right) +\left( \beta _{1}^{-}-%
	\frac{\gamma _{1}}{2\alpha }\right) \frac{R_{0}\left( p_{1}\right) }{2\alpha
	}\cos \frac{\beta \left( x\right) }{\alpha }\int_{0}^{x}\left( q\left(
	t\right) +\frac{p^{2}\left( t\right) }{\alpha ^{2}}\right) dt}%
\end{array}
\label{17}
\end{equation}%
\begin{equation}
\begin{array}{l}
{B\left( x,\varsigma ^{-}\left( x\right) +0\right) -B\left( x,\varsigma
	^{-}\left( x\right) -0\right) =} \\
{\left( \beta _{1}^{-}-\frac{\gamma _{1}}{2\alpha }\right) \frac{R_{0}\left(
		p_{1}\right) }{2\alpha ^{2}}\cos \frac{\beta \left( x\right) }{\alpha }%
	\left( p\left( x\right) -p\left( 0\right) \right) -\left( \beta _{1}^{-}-%
	\frac{\gamma _{1}}{2\alpha }\right) \frac{R_{0}\left( p_{1}\right) }{2\alpha
	}\sin \frac{\beta \left( x\right) }{\alpha }\int_{0}^{x}\left( q\left(
	t\right) +\frac{p^{2}\left( t\right) }{\alpha ^{2}}\right) dt}%
\end{array}
\label{18}
\end{equation}%
\begin{equation}
B\left( x,0\right) =\left. \frac{\partial A\left( x,t\right) }{\partial t}%
\right\vert _{t=0}=0  \label{19}
\end{equation}%
\begin{equation}
\begin{array}{l}
{A\left( x,s^{-}\left( x\right) +0\right) -A\left( x,s^{-}\left( x\right)
	-0\right) =} \\
{-\left( \beta _{2}^{-}-\frac{\gamma _{2}}{2\beta }\right) \frac{R_{2}\left(
		p_{2}\right) }{2\beta ^{2}}\left( p\left( x\right) -p\left( 0\right) \right)
	\,\sin \frac{\omega \left( x\right) }{\beta }-\left( \beta _{2}^{-}-\frac{%
		\gamma _{2}}{2\beta }\right) \frac{R_{2}\left( p_{2}\right) }{2\beta }%
	\int_{0}^{x}\left( q\left( t\right) +\frac{p^{2}\left( t\right) }{\beta ^{2}}%
	\right) dt\,\cos \frac{\omega \left( x\right) }{\beta }}%
\end{array}
\label{20}
\end{equation}

\begin{equation}
\begin{array}{l}
{B\left( x,s^{-}\left( x\right) +0\right) -B\left( x,s^{-}\left( x\right)
	-0\right) =} \\
{-\left( \beta _{2}^{-}-\frac{\gamma _{2}}{2\beta }\right) \frac{R_{2}\left(
		p_{2}\right) }{2\beta ^{2}}\left( p\left( x\right) -p\left( 0\right) \right)
	\,\cos \frac{\omega \left( x\right) }{\beta }+\left( \beta _{2}^{-}-\frac{%
		\gamma _{2}}{2\beta }\right) \frac{R_{2}\left( p_{2}\right) }{2\beta }%
	\int_{0}^{x}\left( q\left( t\right) +\frac{p^{2}\left( t\right) }{\beta ^{2}}%
	\right) dt\,\sin \frac{\omega \left( x\right) }{\beta }}%
\end{array}
\label{21}
\end{equation}%
\begin{equation}
\begin{array}{l}
{A\left( x,s^{+}\left( x\right) +0\right) -A\left( x,s^{+}\left( x\right)
	-0\right) =} \\
{-\left( \beta _{2}^{-}-\frac{\gamma _{2}}{2\beta }\right) \frac{R_{1}\left(
		p_{2}\right) }{2\beta ^{2}}\left( p\left( x\right) -p\left( 0\right) \right)
	\,\sin \frac{\omega \left( x\right) }{\beta }-\left( \beta _{2}^{-}-\frac{%
		\gamma _{2}}{2\beta }\right) \frac{R_{1}\left( p_{2}\right) }{2\beta }%
	\int_{0}^{x}\left( q\left( t\right) +\frac{p^{2}\left( t\right) }{\beta ^{2}}%
	\right) dt\,\cos \frac{\omega \left( x\right) }{\beta }}%
\end{array}
\label{22}
\end{equation}%
\begin{equation}
\begin{array}{l}
{B\left( x,s^{+}\left( x\right) +0\right) -B\left( x,s^{+}\left( x\right)
	-0\right) =} \\
{-\left( \beta _{2}^{-}-\frac{\gamma _{2}}{2\beta }\right) \frac{R_{1}\left(
		p_{2}\right) }{2\beta ^{2}}\left( p\left( x\right) -p\left( 0\right) \right)
	\cos \frac{\omega \left( x\right) }{\beta }\,+\left( \beta _{2}^{-}-\frac{%
		\gamma _{2}}{2\beta }\right) \frac{R_{1}\left( p_{2}\right) }{2\beta }%
	\int_{0}^{x}\left( q\left( t\right) +\frac{p^{2}\left( t\right) }{\beta ^{2}}%
	\right) dt\,\sin \frac{\omega \left( x\right) }{\beta }}%
\end{array}
\label{23}
\end{equation}%
\begin{equation}
\begin{array}{l}
{A\left( x,b^{-}\left( x\right) +0\right) -A\left( x,b^{-}\left( x\right)
	-0\right) =} \\
{-\left( \beta _{2}^{-}+\frac{\gamma _{2}}{2\beta }\right) \frac{R_{2}\left(
		p_{2}\right) }{2\beta ^{2}}\left( p\left( x\right) -p\left( 0\right) \right)
	\,\sin \frac{\omega \left( x\right) }{\beta }-\left( \beta _{2}^{-}-\frac{%
		\gamma _{2}}{2\beta }\right) \frac{R_{2}\left( p_{2}\right) }{2\beta }%
	\int_{0}^{x}\left( q\left( t\right) +\frac{p^{2}\left( t\right) }{\beta ^{2}}%
	\right) dt\cos \frac{\omega \left( x\right) }{\beta }\,}%
\end{array}
\label{24}
\end{equation}%
\begin{equation}
\begin{array}{l}
{B\left( x,b^{-}\left( x\right) +0\right) -B\left( x,b^{-}\left( x\right)
	-0\right) =} \\
{\left( \beta _{2}^{-}+\frac{\gamma _{2}}{2\beta }\right) \frac{R_{2}\left(
		p_{2}\right) }{2\beta ^{2}}\left( p\left( x\right) -p\left( 0\right) \right)
	\,\cos \frac{\omega \left( x\right) }{\beta }-\left( \beta _{2}^{-}-\frac{%
		\gamma _{2}}{2\beta }\right) \frac{R_{2}\left( p_{2}\right) }{2\beta }%
	\int_{0}^{x}\left( q\left( t\right) +\frac{p^{2}\left( t\right) }{\beta ^{2}}%
	\right) dt\,\sin \frac{\omega \left( x\right) }{\beta }}%
\end{array}
\label{25}
\end{equation}%
\begin{equation}
\begin{array}{l}
{A\left( x,b^{+}\left( x\right) +0\right) -A\left( x,b^{+}\left( x\right)
	-0\right) =} \\
{-\left( \beta _{2}^{+}+\frac{\gamma _{2}}{2\beta }\right) \frac{R_{1}\left(
		p_{2}\right) }{2\beta ^{2}}\left( p\left( x\right) -p\left( 0\right) \right)
	\,\sin \frac{\omega \left( x\right) }{\beta }-\left( \beta _{2}^{+}+\frac{%
		\gamma _{2}}{2\beta }\right) \frac{R_{1}\left( p_{2}\right) }{2\beta }%
	\int_{0}^{x}\left( q\left( t\right) +\frac{p^{2}\left( t\right) }{\beta ^{2}}%
	\right) dt\,\cos \frac{\omega \left( x\right) }{\beta }}%
\end{array}
\label{26}
\end{equation}%
\begin{equation}
\begin{array}{l}
{B\left( x,b^{+}\left( x\right) +0\right) -B\left( x,b^{+}\left( x\right)
	-0\right) =} \\
{\left( \beta _{2}^{+}+\frac{\gamma _{2}}{2\beta }\right) \frac{R_{1}\left(
		p_{2}\right) }{2\beta ^{2}}\left( p\left( x\right) -p\left( 0\right) \right)
	\,\cos \frac{\omega \left( x\right) }{\beta }-\left( \beta _{2}^{+}+\frac{%
		\gamma _{2}}{2\beta }\right) \frac{R_{1}\left( p_{2}\right) }{2\beta }%
	\int_{0}^{x}\left( q\left( t\right) +\frac{p^{2}\left( t\right) }{\beta ^{2}}%
	\right) dt\,\sin \frac{\omega \left( x\right) }{\beta }}%
\end{array}
\label{27}
\end{equation}%
are held.

\noindent If in addition we suppose that $p\left(x\right)\in W_{2}^{2}
\left(0,\pi \right),\, q\left(x\right)\in W_{2}^{1} \left(0,\pi \right)$,
the functions

\noindent $A\left( x,t\right) $, $B\left( x,t\right) $ the following system
are provided.
\begin{equation}
\left\{
\begin{array}{l}
{\frac{\partial ^{2}A\left( x,t\right) }{\partial x^{2}}-q\left( x\right)
	A\left( x,t\right) -2p\left( x\right) \frac{\partial B\left( x,t\right) }{%
		\partial t}=\eta \frac{\partial ^{2}A\left( x,t\right) }{\partial t^{2}}} \\
{\frac{\partial ^{2}B\left( x,t\right) }{\partial x^{2}}-q\left( x\right)
	B\left( x,t\right) +2p\left( x\right) \frac{\partial A\left( x,t\right) }{%
		\partial t}=\eta \frac{\partial ^{2}B\left( x,t\right) }{\partial t^{2}}}%
\end{array}%
\right. \,  \label{28}
\end{equation}%
where $\eta =\left\{
\begin{array}{l}
{\alpha ^{2};\,\,\,\,\,\,\,\,\,p_{1}<x<p_{2}} \\
{\beta ^{2};\,\,\,\,\,\,\,\,\,p_{2}<x<\pi }%
\end{array}%
\right. \,\,\,\,\,\,$.\\
The proof is done as in \cite{Ergun-1}.

\noindent Conversely, if the second order derivatives of functions $%
A\left(x,t\right)$, $B\left(x,t\right)$ are summable on $\left[0,\pi \right]$ and $A\left(x,t\right)$, $%
B\left(x,t\right)$ provides $\left(2.22\right)$ system and equations $\left(2.9\right)$-$%
\left(2.21\right)$, then the function $\varphi \left(x,\lambda \right)$ which
is defined by $\left(1.3\right)$-$\left(1.6\right)$ is a solution of $%
\left(1.1\right)$ satisfying boundary conditions $\left(1.2\right)$.

\begin{definition}
	\noindent\ If there is a nontrivial solution $y_{0}\left( x\right) $ that
	provides the $\left( 1.2\right) $ conditions for the $\left( 1.1\right) $
	problem, then $\lambda _{0}$ is called eigenvalue. Additionally, $%
	y_{0}\left( x\right) $ is called the eigenfunction of the problem
	corresponding to the eigenvalue $\lambda _{0}$.
\end{definition}

\noindent Let us assume that $q\left( x\right) $ satisfies the following
conditions.
\begin{equation}
\int_{0}^{\pi }\left\{ \left\vert y^{\prime }\left( x\right) \right\vert
^{2}+q\left( x\right) \left\vert y\left( x\right) \right\vert ^{2}\right\}
dx>0  \label{29}
\end{equation}%
For all $y\left( x\right) \in W_{2}^{2}\left[ 0,\pi \right] $ such that $%
y\left( x\right) \neq 0$ and $y^{\prime }\left( 0\right) \cdot \overline{%
	y\left( 0\right) }-y^{\prime }\left( \pi \right) \cdot \overline{y\left( \pi
	\right) }=0$.

\begin{lemma}
	\noindent The eigenvalues of the  baundary value problem $L$  are real.
\end{lemma}

\begin{proof} We set $l\left(y\right):=\left[-y^{\prime \prime
}+q\left(x\right)y\right]$. Integration by part yields
$$
\begin{array}{l}
{\left(l\left(y\right),y\right)=\int _{0}^{\pi }l\left(y\right)\cdot
	\overline{y\left(x\right)} dx} \\
{\, \, \, \, \, \, \, \, \, \, \, \, \, \, \, \, \, \, \, \, =\int _{0}^{\pi
	}\left\{\left|y^{\prime }\left(x\right)\right|^{2}
	+q\left(x\right)\left|y\left(x\right)\right|^{2} \right\}dx }%
\end{array}%
$$
Since condition $\left(2.23\right)$ holds, it follow that $\left(l\left(y%
\right),y\right)>0$.
\end{proof}
\noindent

\begin{lemma}
	\noindent Eigenfunction corresponding to different eigenvalues of problem $L$ 
	are orthogonal in the sense of the equality
	\begin{equation}
	\left( \lambda _{n}+\lambda _{k}\right) \int_{0}^{\pi }\delta \left(
	x\right) y\left( x,\lambda _{n}\right) y\left( x,\lambda _{k}\right)
	dx-2\int_{0}^{\pi }p\left( x\right) y\left( x,\lambda _{n}\right) y\left(
	x,\lambda _{k}\right) dx=0  \label{30}
	\end{equation}%
	The proof of Lemma 2 carried out as claim \cite{Guseinov-2}.
\end{lemma}
\section{Properties of the spectrum}

Let $\psi \left(x,\lambda \right)$ and $\varphi \left(x,\lambda \right)$ be
any two solutions of equation $\left(1.1\right)$,

\noindent $W\left[\psi \left(x,\lambda \right),\varphi \left(x,\lambda
\right)\right]=\psi \left(x,\lambda \right)\varphi ^{\prime }\left(x,\lambda
\right)-\psi ^{\prime }\left(x,\lambda \right)\varphi \left(x,\lambda \right)
$, Wronskian dosen't depend on $x$. In this case, it depends only on the $%
\lambda $ parameter. Although it is shown as $W\left[\psi ,\varphi \right]%
=\Delta \left(\lambda \right)$. $\Delta \left(\lambda \right)$ is called the
characteristic function of $L$ . Clearly, the function $\Delta \left(\lambda
\right)$ is entire in $\lambda $. It follows that, $\Delta \left(\lambda
\right)$ has at most a countable set of zeros $\left\{\lambda _{n}^{}
\right\}$.

\begin{lemma}
	\noindent The zeros $\left\{ \lambda _{n}\right\} $ of the characteristic
	function $\Delta \left( \lambda \right) $ coincide with the eigenvalues of
	the baundary value problem $L$ . The functions $\psi \left( x,\lambda
	_{0}\right) $ and $\varphi \left( x,\lambda _{0}\right) $ are eigenfunctions
	corresponding to the eigenvalue $\lambda _{n}$, and there exist a sequence $%
	\left( \beta _{n}\right) $ such that
	\begin{equation}
	\psi \left( x,\lambda _{n}\right) =\beta _{n}\varphi \left( x,\lambda
	_{n}\right) \,\,,\,\,\beta _{n}\neq 0  \label{31}
	\end{equation}%
	The proof of the Lemma 3 is done as in \cite{Yurko}.
\end{lemma}

\noindent Let use denote
\begin{equation}
\alpha _{n}=\int_{0}^{\pi }\delta \left( x\right) \varphi ^{2}\left(
x,\lambda _{n}\right) dx-\frac{1}{\lambda _{n}}\int_{0}^{\pi }p\left(
x\right) \varphi ^{2}\left( x,\lambda _{n}\right) dx,\,\,n=1,2,3,...
\label{32}
\end{equation}%
The numbers $\left\{ \alpha _{n}\right\} $ are called normalized numbers of
the problem $L$ .

\begin{lemma}
	\noindent The equality $\mathop{\Delta }\limits^{\bullet }\left( \lambda
	_{n}\right) =2\lambda _{n}\beta _{n}\alpha _{n}$ is obtained. Here $%
	\mathop{\Delta }\limits^{\bullet }=\frac{d\Delta }{d\lambda }$.
\end{lemma}

\begin{proof}
	\noindent Since $\varphi \left( x,\lambda \right) $ and $\psi \left(
	x,\lambda \right) $ are the solutions of $\left( 1.1\right) $,
	$$
	-\varphi ^{\prime \prime }\left( x,\lambda \right) +\left[ 2\lambda p\left(
	x\right) +q\left( x\right) \right] \varphi \left( x,\lambda \right) =\lambda
	^{2}\delta \left( x\right) \varphi \left( x,\lambda \right)
	$$
	$$
	-\psi ^{\prime \prime }\left( x,\lambda \right) +\left[ 2\lambda p\left(
	x\right) +q\left( x\right) \right] \psi \left( x,\lambda \right) =\lambda
	^{2}\delta \left( x\right) \psi \left( x,\lambda \right)
	$$
	equations are provided. Hence, we differentiate the equalities with respect
	to
	$$
	-\mathop{\varphi ''}\limits^{\bullet }\left( x,\lambda \right) +\left[
	2\lambda p\left( x\right) +q\left( x\right) \right] \mathop{\varphi }\limits%
	^{\bullet }\left( x,\lambda \right) =\lambda ^{2}\delta \left( x\right) %
	\mathop{\varphi }\limits^{\bullet }\left( x,\lambda \right) +\left[ 2\lambda
	\delta \left( x\right) -2p\left( x\right) \right] \varphi \left( x,\lambda
	\right)
	$$
	$$
	-\mathop{\psi ''}\limits^{\bullet }\left( x,\lambda \right) +\left[ 2\lambda
	p\left( x\right) +q\left( x\right) \right] \mathop{\psi }\limits^{\bullet
	}\left( x,\lambda \right) =\lambda ^{2}\delta \left( x\right) \mathop{\psi }%
	\limits^{\bullet }\left( x,\lambda \right) +\left[ 2\lambda \delta \left(
	x\right) -2p\left( x\right) \right] \psi \left( x,\lambda \right) .
	$$
	Thanks to these equations
	$$
	\begin{array}{l}
	{\frac{d}{dx}\left\{ \varphi \left( x,\lambda \right) \cdot \mathop{\psi '}%
		\limits^{\bullet }\left( x,\lambda \right) -\varphi ^{\prime }\left(
		x,\lambda \right) \cdot \mathop{\psi }\limits^{\bullet }\left( x,\lambda
		\right) \right\} =-\left[ 2\lambda \delta \left( x\right) -2p\left( x\right) %
		\right] \varphi \left( x,\lambda \right) \psi \left( x,\lambda \right) } \\
	{\frac{d}{dx}\left\{ \mathop{\varphi }\limits^{\bullet }\left( x,\lambda
		\right) \cdot \psi ^{\prime }\left( x,\lambda \right) -\mathop{\varphi '}%
		\limits^{\bullet }\left( x,\lambda \right) \cdot \psi \left( x,\lambda
		\right) \right\} =\left[ 2\lambda \delta \left( x\right) -2p\left( x\right) %
		\right] \varphi \left( x,\lambda \right) \psi \left( x,\lambda \right). }%
	\end{array}%
	$$
	If the last equations are integrated from $x$ to $\pi $ and from 0 to $x$, 
	respectively, by the discontinuity conditions, we obtain
	$$
	-\left. \left\{ \varphi \left( \xi ,\lambda \right) \cdot \mathop{\psi '}%
	\limits^{\bullet }\left( \xi ,\lambda \right) -\varphi ^{\prime }\left( \xi
	,\lambda \right) \cdot \mathop{\psi }\limits^{\bullet }\left( \xi ,\lambda
	\right) \right\} \right\vert _{x}^{\pi }=\int_{x}^{\pi }\left[ 2\lambda
	\delta \left( \xi \right) -2p\left( \xi \right) \right] \varphi \left( \xi
	,\lambda \right) \psi \left( \xi ,\lambda \right) d\xi
	$$
	and
	$$
	\left. \left\{ \mathop{\varphi }\limits^{\bullet }\left( \xi ,\lambda
	\right) \cdot \psi ^{\prime }\left( \xi ,\lambda \right) -\mathop{\varphi '}%
	\limits^{\bullet }\left( \xi ,\lambda \right) \cdot \psi \left( \xi ,\lambda
	\right) \right\} \right\vert _{0}^{x}=\int_{0}^{x}\left[ 2\lambda \delta
	\left( \xi \right) -2p\left( \xi \right) \right] \varphi \left( \xi ,\lambda
	\right) \psi \left( \xi ,\lambda \right) d\xi.
	$$
	If we add the last equalities side by side, we get
	\[\begin{array}{l} {W\left[\varphi \left(\xi ,\lambda \right),\mathop{\psi }\limits^{.} \left(\xi ,\lambda \right)\right]+W\left[\mathop{\varphi }\limits^{.} \left(\xi ,\lambda \right),\psi \left(\xi ,\lambda \right)\right]} \\ {=-\mathop{\Delta }\limits^{.} \left(\lambda \right)=\int _{0}^{\pi }\left[2\lambda \delta \left(\xi \right)-2p\left(\xi \right)\right] \varphi \left(\xi ,\lambda \right)\psi \left(\xi ,\lambda \right)d\xi } \end{array}\] 
	for $\lambda \to \lambda _{n} $, this yields
	$$
	\begin{array}{l}
	{\mathop{\Delta }\limits^{\bullet }\left( \lambda _{n}\right)
		=-\int_{0}^{\pi }\left[ 2\lambda _{n}\delta \left( \xi \right) -2p\left( \xi
		\right) \right] \beta _{n}\varphi ^{2}\left( \xi ,\lambda _{n}\right) d\xi }
	\\
	{=2\lambda _{n}\beta _{n}\left\{ \int_{0}^{\pi }\delta \left( \xi \right)
		\varphi ^{2}\left( \xi ,\lambda _{n}\right) d\xi -\frac{1}{\lambda _{n}}%
		\int_{0}^{\pi }p\left( \xi \right) \varphi ^{2}\left( \xi ,\lambda
		_{n}\right) d\xi \right\} =2\lambda _{n}\beta _{n}\alpha _{n}.}%
	\end{array}%
	$$
	Denote,\[\Gamma _{n} =\left\{\lambda :\left|\lambda \right|=\left|\lambda _{n}^{0} \right|+\delta ,\delta >0,n=0,1,2,...\right\}, G_{n} =\left\{\lambda :\left|\lambda -\lambda _{n}^{0} \right|\ge \delta ,\delta >0,n=0,1,2,...\right\}\] 
	where $\delta $ is sufficiently small positive number.
	
	For sufficiently large values of $n$, one has\\
	
\begin{equation}
\left\vert \Delta \left(
\lambda \right) -\Delta _{0}\left( \lambda \right) \right\vert <\frac{%
	C_{\delta }}{2}e^{\left\vert \tau \right\vert \left( \beta \pi -\beta
	a_{2}+\alpha p_{2}-\alpha p_{1}+p_{1}\right) },\lambda \in \Gamma _{n}.
\end{equation} 
	
	As it is shown in \cite{Levin} , $\left\vert \Delta _{0}\left( \lambda
	\right) \right\vert \geq C_{\delta }e^{\left\vert Im\lambda \right\vert \pi }
	$ for all $\lambda \in \bar{G}_{\delta }$ , where $C_{\delta }>0$
	\[\begin{array}{l} {\mathop{\lim }\limits_{\left|\lambda \right|\to \infty } e^{-\left|Im\lambda \right|\pi } \left(\Delta \left(\lambda \right)-\Delta _{0} \left(\lambda \right)\right)} \\ {=\mathop{\lim }\limits_{\left|\lambda \right|\to \infty } e^{-\left|Im\lambda \right|\pi } \left(\int _{0}^{\pi }\tilde{A}\left(\pi ,t\right)\cos \lambda tdt +\int _{0}^{\pi }\tilde{B}\left(\pi ,t\right)\sin \lambda tdt \right)=0} \end{array}\]
	is constant. On the other hand, since for sufficiently large values of $n$ (see\cite{Marchenko}) we get $\left(
	3.3\right) $. The Lemma 4 is proved.
\end{proof}

\noindent

\begin{lemma}
	The problem $L\left( \alpha ,p_{1},p_{2}\right) $ has countable set of
	eigenvalues. If one denotes by $\lambda _{1},\lambda _{2},...$ the
	positive eigenvalues arranged in increasing order and by $\lambda
	_{-1},\lambda _{-2},...$ the negative eigenvalues arranged in decreasing
	order, then eigenvalues of the problem $L\left( \alpha ,p_{1},p_{2}\right) $ 
	have the asymptotic behavior
	$$
	\lambda _{n}=\lambda _{n}^{0}+\frac{d_{n}}{\lambda _{n}^{0}}+\frac{k_{n}}{%
		\lambda _{n}^{0}}\,\,,\,\,n\rightarrow \infty
	$$
	where $k_{n}\in l_{2}$, $d_{n}$ is a bounded sequence and\\$\lambda _{n}^{0}=%
	\frac{n\pi }{\beta \pi -\beta p_{2}+\alpha p_{2}-\alpha p_{1}+p_{1}}+\psi
	_{1}\left( n\right) ;\,\,\,\,\mathop{\sup }\limits_{n}\left\vert \psi
	_{1}\left( n\right) \right\vert =c<+\infty $.
\end{lemma}

\begin{proof}
	\noindent According to previous lemma, if $n$ is a sufficiently large
	natural number and $\lambda \in \Gamma _{n}$, we have $\left\vert \Delta
	_{0}\left( \lambda \right) \right\vert \geq C_{\delta }e^{\left\vert
		Im\lambda \right\vert \pi }>\frac{C_{\delta }}{2}e^{\left\vert Im\lambda
		\right\vert \pi }>\left\vert \Delta \left( \lambda \right) -\Delta
	_{0}\left( \lambda \right) \right\vert .$ Applying Rouche's theorem, we
	conclude that for sufficiently large $n$ inside the contour $\Gamma _{n}$ the
	functions $\Delta _{0}\left( \lambda \right) $ and $\Delta _{0}\left( \lambda
	\right) +\left\{ \Delta \left( \lambda \right) -\Delta _{0}\left( \lambda
	\right) \right\} =\Delta \left( \lambda \right) $ have the same number of
	zeros. That is, there are exactly $\left( n+1\right) $ zeros $\lambda
	_{1},\lambda _{2},...,\lambda _{n}$. Analogously, it is shown by Rouche's
	theorem that, for sufficiently large values of $n$, the function $\Delta
	\left( \lambda \right) $ has a unique zero inside each circle $\left\vert
	\lambda -\lambda _{n}^{0}\right\vert <\delta $ . Since $\delta >0$ is a
	arbitrary, it follows that $\lambda _{n}=\lambda _{n}^{0}+\varepsilon _{n}$
	, where $\mathop{\lim }\limits_{n\rightarrow \infty }\varepsilon _{n}=0$ .
	If $\Delta \left( \lambda _{n}\right) =0$, we have
	\begin{equation}
	\Delta _{0}\left( \lambda _{n}^{0}+\varepsilon _{n}\right) +\int_{0}^{\pi
	}A\left( \pi ,t\right) \cos \left( \lambda _{n}^{0}+\varepsilon _{n}\right)
	tdt+\int_{0}^{\pi }B\left( \pi ,t\right) \sin \left( \lambda
	_{n}^{0}+\varepsilon _{n}\right) tdt=0  \label{34}
	\end{equation}%
	\begin{equation}
	\begin{array}{l}
	{\Delta _{0}\left( \lambda _{n}^{0}+\varepsilon _{n}\right) =\left( \beta
		_{2}^{+}+\frac{\gamma _{2}}{2\beta }\right) R_{1}\left( p_{2}\right) \cos %
		\left[ \left( \lambda _{n}^{0}+\varepsilon _{n}\right) b^{+}\left( \pi
		\right) -\frac{1}{\beta }\int_{p_{2}}^{\pi }p\left( t\right) dt\right] } \\
	{\,\,\,\,\,\,\,\,\,\,\,\,\,\,\,\,\,\,\,\,\,\,\,\,\,\,\,+\left( \beta
		_{2}^{-}+\frac{\gamma _{2}}{2\beta }\right) R_{2}\left( p_{2}\right) \cos %
		\left[ \left( \lambda _{n}^{0}+\varepsilon _{n}\right) b^{-}\left( \pi
		\right) -\frac{1}{\beta }\int_{p_{2}}^{\pi }p\left( t\right) dt\right] } \\
	{\,\,\,\,\,\,\,\,\,\,\,\,\,\,\,\,\,\,\,\,\,\,\,\,\,\,\,+\left( \beta
		_{2}^{-}-\frac{\gamma _{2}}{2\beta }\right) R_{1}\left( p_{2}\right) \cos %
		\left[ \left( \lambda _{n}^{0}+\varepsilon _{n}\right) s^{+}\left( \pi
		\right) +\frac{1}{\beta }\int_{p_{2}}^{\pi }p\left( t\right) dt\right] } \\
	{\,\,\,\,\,\,\,\,\,\,\,\,\,\,\,\,\,\,\,\,\,\,\,\,\,\,\,+\left( \beta
		_{2}^{+}-\frac{\gamma _{2}}{2\beta }\right) R_{2}\left( p_{2}\right) \cos %
		\left[ \left( \lambda _{n}^{0}+\varepsilon _{n}\right) s^{-}\left( \pi
		\right) +\frac{1}{\beta }\int_{p_{2}}^{\pi }p\left( t\right) dt\right] }%
	\end{array}
	\label{35}
	\end{equation}%
	Since $\Delta _{0}\left( \lambda \right) $ is an analytical function,\\ $%
	\Delta _{0}\left( \lambda _{n}^{0}+\varepsilon _{n}\right) =\Delta
	_{0}\left( \lambda _{n}^{0}\right) \varepsilon _{n}+\mathop{\Delta }\limits%
	_{0}^{\cdot }\left( \lambda _{n}^{0}\right) \varepsilon _{n}+\frac{%
		\mathop{\Delta }\limits_{0}^{\cdot \cdot }\left( \lambda _{n}^{0}\right) }{2!%
	}\varepsilon _{n}^{2}+...\,\,,\,\mathop{\lim }\limits_{n\rightarrow \infty
	}\varepsilon _{n}=0$.
	
	$\lambda _{n}^{0}$ is the roots of the $\Delta _{0}\left( \lambda \right) =0$
	equation
	
	$\Delta _{0}\left( \lambda _{n}^{0}+\varepsilon _{n}\right) =\left[ %
	\mathop{\Delta }\limits_{0}^{.}\left( \lambda _{n}^{0}\right) +o\left(
	1\right) \right] \varepsilon _{n}\,\,,\,\,n\rightarrow \infty $ is provided.
	\[\begin{array}{l} {\left[\mathop{\Delta }\limits^{.} _{0} \left(\lambda _{n} ^{0} \right)+o\left(1\right)\right]\varepsilon _{n} \, } \\ {+\int _{p_{2} }^{s^{-} \left(x\right)-0}A\left(\pi ,t\right)\cos \left(\lambda _{n} ^{0} +\varepsilon _{n} \right)tdt +\int _{s^{-} \left(x\right)+0}^{s^{+} \left(x\right)-0}A\left(\pi ,t\right)\cos \left(\lambda _{n} ^{0} +\varepsilon _{n} \right)tdt } \\ {+\int _{s^{+} \left(x\right)+0}^{b^{-} \left(x\right)-0}A\left(\pi ,t\right)\cos \left(\lambda _{n} ^{0} +\varepsilon _{n} \right)tdt +\int _{b^{-} \left(x\right)+0}^{b^{+} \left(x\right)-0}A\left(\pi ,t\right)\cos \left(\lambda _{n} ^{0} +\varepsilon _{n} \right)tdt } \\ {+\int _{b^{+} \left(x\right)+0}^{x}A\left(\pi ,t\right)\cos \left(\lambda _{n} ^{0} +\varepsilon _{n} \right)tdt +\int _{p_{2} }^{s^{-} \left(x\right)-0}B\left(\pi ,t\right)\sin \left(\lambda _{n} ^{0} +\varepsilon _{n} \right)tdt } \\ {+\int _{s^{-} \left(x\right)+0}^{s^{+} \left(x\right)-0}B\left(\pi ,t\right)\sin \left(\lambda _{n} ^{0} +\varepsilon _{n} \right)tdt +\int _{s^{+} \left(x\right)+0}^{b^{-} \left(x\right)-0}B\left(\pi ,t\right)\sin \left(\lambda _{n} ^{0} +\varepsilon _{n} \right)tdt } \\ {+\int _{b^{-} \left(x\right)+0}^{b^{+} \left(x\right)-0}B\left(\pi ,t\right)\sin \left(\lambda _{n} ^{0} +\varepsilon _{n} \right)tdt +\int _{b^{+} \left(x\right)+0}^{x}B\left(\pi ,t\right)\sin \left(\lambda _{n} ^{0} +\varepsilon _{n} \right)tdt =0} \end{array}\]
	\ It is easy to see that the function $\Delta _{0}\left( \lambda \right) =0
	$ is type of \cite{Jdanovich}, so there is a $\eta _{\delta }>0$ such that
	$\left\vert \mathop{\Delta }\limits_{0}^{.}\left( \lambda _{n}^{0}\right)
	\right\vert \geq \eta _{\delta }>0$ is satisfied for all $n$. We also have
	\begin{equation}
	\lambda _{n}^{0}=\frac{n\pi }{\beta \pi -\beta p_{2}+\alpha p_{2}-\alpha
		p_{1}+p_{1}}+\psi _{1}\left( n\right)   \label{36}
	\end{equation}%
	where $\mathop{\sup }\limits_{n}\left\vert \psi _{1}\left( n\right)
	\right\vert <M$ is for some constant $M>0$ \cite{Krein}. Further,
	substituting $\left( 3.6\right) $ into $\left( 3.5\right) $ after certain
	transformations, we reach $\varepsilon _{n}\in l_{2}$ . We can obtain more
	precisely
	
	Since $\left( \int_{0}^{\pi }A_{t}\left( \pi ,t\right) \sin \left( \lambda
	_{n}^{0}+\varepsilon _{n}\right) tdt\right) \in l_{2}$ and $\left(
	\int_{0}^{\pi }B_{t}\left( \pi ,t\right) \cos \left( \lambda
	_{n}^{0}+\varepsilon _{n}\right) tdt\right) \in l_{2}$ , we have
	\[\begin{array}{l} {\varepsilon _{n} =\frac{1}{2\lambda _{n} ^{0} \Delta _{0} \left(\lambda _{n} ^{0} \right)} \left\{\left[\left(\beta _{2} ^{-} -\frac{\gamma _{2} }{2\beta } \right)\frac{R_{2} \left(p_{2} \right)}{2\beta } \sin \left[\lambda _{n} ^{0} s^{-} \left(\pi \right)+\frac{\omega \left(x\right)}{\beta } \right]\right. \right. } \\ {+\left(\beta _{2} ^{-} -\frac{\gamma _{2} }{2\beta } \right)\frac{R_{1} \left(p_{2} \right)}{2\beta } \sin \left[\lambda _{n} ^{0} s^{+} \left(\pi \right)+\frac{\omega \left(x\right)}{\beta } \right]+\left(\beta _{2} ^{-} -\frac{\gamma _{2} }{2\beta } \right)\frac{R_{2} \left(p_{2} \right)}{2\beta } \sin \left[\lambda _{n} ^{0} b^{-} \left(\pi \right)-\frac{\omega \left(x\right)}{\beta } \right]} \\ {\left. +\left(\beta _{2} ^{+} +\frac{\gamma _{2} }{2\beta } \right)\frac{R_{1} \left(p_{2} \right)}{2\beta } \sin \left[\lambda _{n} ^{0} b^{+} \left(\pi \right)-\frac{\omega \left(x\right)}{\beta } \right]\right]\int _{0}^{\pi }\left(q\left(t\right)+p^{2} \left(t\right)\right)dt } \\ {+\left[-\left(\beta _{2} ^{-} -\frac{\gamma _{2} }{2\beta } \right)\frac{R_{2} \left(p_{2} \right)}{2\beta ^{2} } \cos \left[\lambda _{n} ^{0} s^{-} \left(\pi \right)+\frac{\omega \left(x\right)}{\beta } \right]\right. } \\ {-\left(\beta _{2} ^{-} -\frac{\gamma _{2} }{2\beta } \right)\frac{R_{1} \left(p_{2} \right)}{2\beta ^{2} } \cos \left[\lambda _{n} ^{0} s^{+} \left(\pi \right)+\frac{\omega \left(x\right)}{\beta } \right]+\left(\beta _{2} ^{-} +\frac{\gamma _{2} }{2\beta } \right)\frac{R_{2} \left(p_{2} \right)}{2\beta ^{2} } \cos \left[\lambda _{n} ^{0} b^{-} \left(\pi \right)-\frac{\omega \left(x\right)}{\beta } \right]} \\ {\left. \left. +\left(\beta _{2} ^{+} +\frac{\gamma _{2} }{2\beta } \right)\frac{R_{1} \left(p_{2} \right)}{2\beta ^{2} } \cos \left[\lambda _{n} ^{0} b^{+} \left(\pi \right)-\frac{\omega \left(x\right)}{\beta } \right]\right]\left[p\left(\pi \right)-p\left(0\right)\right]\right\}+\frac{k_{n} }{\lambda _{n} ^{0} } \, } \end{array}\]
	where
	\[\begin{array}{l} {d_{n} =\frac{1}{2\Delta _{0} \left(\lambda _{n} ^{0} \right)} \left\{\left[\left(\beta _{2} ^{-} -\frac{\gamma _{2} }{2\beta } \right)\frac{R_{2} \left(p_{2} \right)}{2\beta } \sin \left[\lambda _{n} ^{0} s^{-} \left(\pi \right)+\frac{\omega \left(x\right)}{\beta } \right]\right. \right. } \\ {+\left(\beta _{2} ^{-} -\frac{\gamma _{2} }{2\beta } \right)\frac{R_{1} \left(p_{2} \right)}{2\beta } \sin \left[\lambda _{n} ^{0} s^{+} \left(\pi \right)+\frac{\omega \left(x\right)}{\beta } \right]+\left(\beta _{2} ^{-} -\frac{\gamma _{2} }{2\beta } \right)\frac{R_{2} \left(p_{2} \right)}{2\beta } \sin \left[\lambda _{n} ^{0} b^{-} \left(\pi \right)-\frac{\omega \left(x\right)}{\beta } \right]} \\ {\left. +\left(\beta _{2} ^{+} +\frac{\gamma _{2} }{2\beta } \right)\frac{R_{1} \left(p_{2} \right)}{2\beta } \sin \left[\lambda _{n} ^{0} b^{+} \left(\pi \right)-\frac{\omega \left(x\right)}{\beta } \right]\right]\int _{0}^{\pi }\left(q\left(t\right)+p^{2} \left(t\right)\right)dt } \\ {+\left[-\left(\beta _{2} ^{-} -\frac{\gamma _{2} }{2\beta } \right)\frac{R_{2} \left(p_{2} \right)}{2\beta ^{2} } \cos \left[\lambda _{n} ^{0} s^{-} \left(\pi \right)+\frac{\omega \left(x\right)}{\beta } \right]\right. } \\ {-\left(\beta _{2} ^{-} -\frac{\gamma _{2} }{2\beta } \right)\frac{R_{1} \left(p_{2} \right)}{2\beta ^{2} } \cos \left[\lambda _{n} ^{0} s^{+} \left(\pi \right)+\frac{\omega \left(x\right)}{\beta } \right]+\left(\beta _{2} ^{-} +\frac{\gamma _{2} }{2\beta } \right)\frac{R_{2} \left(p_{2} \right)}{2\beta ^{2} } \cos \left[\lambda _{n} ^{0} b^{-} \left(\pi \right)-\frac{\omega \left(x\right)}{\beta } \right]} \\ {\left. \left. +\left(\beta _{2} ^{+} +\frac{\gamma _{2} }{2\beta } \right)\frac{R_{1} \left(p_{2} \right)}{2\beta ^{2} } \cos \left[\lambda _{n} ^{0} b^{+} \left(\pi \right)-\frac{\omega \left(x\right)}{\beta } \right]\right]\left[p\left(\pi \right)-p\left(0\right)\right]\right\}} \end{array}\]
	is bounded sequence. The proof is completed.
\end{proof}

\noindent

	\noindent The $\varphi \left( x,\lambda \right) $ function is $\left\vert
	\lambda \right\vert \rightarrow \infty $ in the region $D=\left\{ \lambda :\arg
	\lambda \in \left[ \varepsilon ,\pi -\varepsilon \right] \right\} $

\noindent for $x>p_{2} $,
$$
\varphi \left(x,\lambda \right)=\frac{1}{2} \left(\beta _{2}^{+} +\frac{%
	\gamma _{2} }{2\beta }  \right)\exp
\left(-i\left(\lambda b^{+}
\left(x\right)-w\left(x\right)\right)\right)\left(1+O\left(\frac{1}{\lambda }
\right)\right),\left|\lambda \right|\to \infty
$$
it has an asymptotic representation where $w\left(x\right)=\int _{p_{2}
}^{x}p\left(t\right)dt $ \newline
and $\beta _{2} ^{\mp } =\frac{1}{2} \left(\alpha _{2} \mp \frac{\alpha
	\beta _{2} }{\beta } \right)$ .

\section{Inverse Problem.}

\noindent Let us consider the baundary value problem $\tilde{L}$ :
$$
\tilde{L}:=\left\{%
\begin{array}{l}
{l\left(y\right):=-y^{\prime \prime }+\left[2\lambda \tilde{p}\left(x\right)+%
	\tilde{q}\left(x\right)\right]y=\lambda ^{2} \tilde{\delta }%
	\left(x\right)y,\, \, x\in \left(0,\pi \right)\, \, \, \, \, \, \, \, \, \,
	\, \, \, \, \, \, \, \, \, \, \, } \\
{y^{\prime }\left(0\right)=0,y\left(\pi
	\right)=0\, \, \, \, \, \, \, \, \, \, \, \, \, \, \, \, \, \, \, \, \, \,
	\, \, \, \, \, \, \, \, \, \, \, \, \, \, \, \, \, \, \, \, \, \, \, \, \,
	\, \, \, \, \, \, \, \, \, \, \, \, \, \, \, \, \, \, \, \, \, } \\
{y\left(\tilde{p}_{1} +0\right)=\tilde{\alpha }_{1} y\left(\tilde{p}_{1}
	-0\right)\, \, \, \, \, \, \, \, \, \, \, \, \, \, \, \, \, \, \, \, \, \,
	\, \, \, \, \, \, \, \, \, \, \, \, \, \, \, \, \, \, \, \, \, \, \, \, \,
	\, \, \, \, \, \, \, \, \, \, \, \, \, \, \, \, \, \, \, \, \, \, \, \, \,
	\, \, \, \, \, \, \, \, \, \, \, \, \, \, \, \, \, \, \, \, \, \, \, \, \, }
\\
{y^{\prime }\left(\tilde{p}_{1} +0\right)=\tilde{\beta }_{1} y^{\prime
	}\left(\tilde{p}_{1} -0\right)+i\lambda \tilde{\gamma }_{1} y\left(\tilde{p}%
	_{1} -0\right)\, \, \, \, \, \, \, \, \, \, \, \, \, \, \, \, \, \, \, \, \,
	\, \, \, \, \, \, \, \, \, \, \, \, \, \, \, \, \, \, \, \, \, \, \, \, \,
	\, \, \, \, \, \, \, \, \, \, \, } \\
{y\left(\tilde{p}_{2} +0\right)=\tilde{\alpha }_{2} y\left(\tilde{p}_{2}
	-0\right)\, \, \, \, \, \, \, \, \, \, \, \, \, \, \, \, \, \, \, \, \, \,
	\, \, \, \, \, \, \, \, \, \, \, \, \, \, \, \, \, \, \, \, \, \, \, \, \,
	\, \, \, \, \, \, \, \, \, \, \, \, \, \, \, \, \, \, \, \, \, \, \, \, \,
	\, \, \, \, \, \, \, \, \, \, \, \, \, \, \, \, \, \, \, \, \, \, \, } \\
{y^{\prime }\left(\tilde{p}_{2} +0\right)=\tilde{\beta }_{2} y^{\prime
	}\left(\tilde{p}_{2} -0\right)+i\lambda \tilde{\gamma }_{2} y\left(\tilde{p}%
	_{2} -0\right)\, \, \, \, \, \, \, \, \, \, \, \, \, \, \, \, \, \, \, \, \,
	\, \, \, \, \, \, \, \, \, \, \, \, \, \, \, \, \, \, \, \, \, \, \, \, \,
	\, \, \, \, \, \, \, \, }%
\end{array}%
\right.
$$
Let the function $\Phi \left(x,\lambda \right)$ denote solution of $%
\left(1.1\right)$ that satisfy the conditions $\Phi '\left(0\right)=1\, ,\, \Phi \left(\pi \right)=0$ respectively and jump conditions $\left(1.3\right)-%
\left(1.6\right)$. Lets define it as $M\left(\lambda \right):=\Phi
\left(0,\lambda \right)$.

\noindent The $\Phi \left(x,\lambda \right)$ and $M\left(\lambda \right)$
functions are called the Weyl solution and the Weyl function, respectively.
\\$\Phi \left( x,\lambda \right) =M\left( \lambda \right)
	.\varphi \left( x,\lambda \right) +S\left( x,\lambda \right) \,\,,\,\,\left(
	\lambda \neq \lambda _{n}\,\,;\,\,n=1,2,3,...\right) $ is true.
Because of $\left. W\left[\varphi ,S\right]\right|_{x=0} =\varphi
\left(0,\lambda \right)S^{\prime }\left(0,\lambda \right)-\varphi ^{\prime
}\left(0,\lambda \right)S\left(0,\lambda \right)=1\ne 0$, $\varphi
\left(x,\lambda \right)$ and $S\left(x,\lambda \right)$ solutions are linear
independent. When $\psi \left(x,\lambda \right)$ is solution $\left(1.1\right)$%
,
$$
\begin{array}{l}
{\psi \left(x,\lambda \right)=A\left(\lambda \right)\varphi \left(x,\lambda
	\right)+B\left(\lambda \right)S\left(x,\lambda \right)} \\
{\psi ^{\prime }\left(x,\lambda \right)=A\left(\lambda \right)\varphi
	^{\prime }\left(x,\lambda \right)+B\left(\lambda \right)S^{\prime
	}\left(x,\lambda \right).}%
\end{array}%
$$
Due to boundary conditions, $A\left(\lambda \right)=\psi \left(0,\lambda
\right)\, \, ,\, \, B\left(\lambda \right)=\psi ^{\prime }\left(0,\lambda
\right)=-\Delta \left(\lambda \right)$ . Then $\psi \left(x,\lambda
\right)=\psi \left(0,\lambda \right)\varphi \left(x,\lambda \right)-\Delta
\left(\lambda \right)S\left(x,\lambda \right)$ is obtained. Hence,\newline
$\Phi \left(x,\lambda \right):=-\frac{\psi \left(x,\lambda \right)}{\Delta
	\left(\lambda \right)} =S\left(x,\lambda \right)+M\left(\lambda
\right)\varphi \left(x,\lambda \right)\, \, ,\, \, M\left(\lambda \right)=-%
\frac{\psi \left(0,\lambda \right)}{\Delta \left(\lambda \right)} $.\newline
The $M\left(\lambda \right)$ function is a meromorphic function.

\begin{theorem}
	\noindent\ If $M\left( \lambda \right) =\tilde{M}\left( \lambda \right) $,
	then $L=\tilde{L}$.
\end{theorem}

\begin{proof}
	\noindent Let us define the matrix $P\left( x,\lambda \right) =\left[
	P_{j,k}\left( x,\lambda \right) \right] ,\left( j,k=1,2\right) $ by the
	formula
	
	$P\left( x,\lambda \right) \cdot \left(
	\begin{array}{cc}
	{\widetilde{\varphi }\left( x,\lambda \right) } & {\widetilde{\Phi }\left(
		x,\lambda \right) } \\
	{\widetilde{\varphi }^{{^{\prime }}}\left( x,\lambda \right) } & {\widetilde{%
			\Phi }^{{^{\prime }}}\left( x,\lambda \right) }%
	\end{array}%
	\right) =\left(
	\begin{array}{cc}
	{\varphi \left( x,\lambda \right) } & {\Phi \left( x,\lambda \right) } \\
	{\varphi ^{{^{\prime }}}\left( x,\lambda \right) } & {\Phi ^{{^{\prime }}%
		}\left( x,\lambda \right) }%
	\end{array}%
	\right) $ .\\ In this case
	
	$$
	\begin{array}{l}
	{P_{11}\left( x,\lambda \right) =-\varphi \left( x,\lambda \right) \frac{%
			\widetilde{\psi }^{{^{\prime }}}\left( x,\lambda \right) }{\widetilde{\Delta
			}\left( \lambda \right) }+\widetilde{\varphi }^{{^{\prime }}}\left(
		x,\lambda \right) \frac{\psi \left( x,\lambda \right) }{\Delta \left(
			\lambda \right) }} \\
	{P_{12}\left( x,\lambda \right) =-\widetilde{\varphi }\left( x,\lambda
		\right) \frac{\psi \left( x,\lambda \right) }{\Delta \left( \lambda \right) }%
		+\varphi \left( x,\lambda \right) \frac{\widetilde{\psi }\left( x,\lambda
			\right) }{\widetilde{\Delta }\left( \lambda \right) }} \\
	{P_{21}\left( x,\lambda \right) =-\varphi ^{\prime }\left( x,\lambda \right)
		\frac{\widetilde{\psi }^{{^{\prime }}}\left( x,\lambda \right) }{\widetilde{%
				\Delta }\left( \lambda \right) }-\widetilde{\varphi }^{{^{\prime }}}\left(
		x,\lambda \right) \frac{\psi ^{\prime }\left( x,\lambda \right) }{\Delta
			\left( \lambda \right) }} \\
	{P_{22}\left( x,\lambda \right) =-\widetilde{\varphi }\left( x,\lambda
		\right) \frac{\psi ^{\prime }\left( x,\lambda \right) }{\Delta \left(
			\lambda \right) }+\varphi ^{\prime }\left( x,\lambda \right) \frac{%
			\widetilde{\psi }\left( x,\lambda \right) }{\widetilde{\Delta }\left(
			\lambda \right) }}%
	\end{array}%
	$$
	Hence,
	$$
	\begin{array}{l}
	{P_{11}\left( x,\lambda \right) =\varphi \left( x,\lambda \right) \left[
		\widetilde{S}^{{^{\prime }}}\left( x,\lambda \right) +\widetilde{M}\left(
		\lambda \right) \cdot \widetilde{\varphi }^{{^{\prime }}}\left( x,\lambda
		\right) \right] -\widetilde{\varphi }^{{^{\prime }}}\left( x,\lambda \right) %
		\left[ S\left( x,\lambda \right) +M\left( \lambda \right) \cdot \varphi
		\left( x,\lambda \right) \right] } \\
	{=\varphi \left( x,\lambda \right) \widetilde{S}^{{^{\prime }}}\left(
		x,\lambda \right) -\widetilde{\varphi }^{{^{\prime }}}\left( x,\lambda
		\right) S\left( x,\lambda \right) +\left[ \widetilde{M}\left( \lambda
		\right) -M\left( \lambda \right) \right] \varphi \left( x,\lambda \right)
		\widetilde{\varphi }^{{^{\prime }}}\left( x,\lambda \right) } \\
	\\
	{P_{12}\left( x,\lambda \right) =\widetilde{\varphi }\left( x,\lambda
		\right) \left[ S\left( x,\lambda \right) +M\left( \lambda \right) \cdot
		\varphi \left( x,\lambda \right) \right] -\varphi \left( x,\lambda \right) %
		\left[ \widetilde{S}\left( x,\lambda \right) +\widetilde{M}\left( \lambda
		\right) \cdot \widetilde{\varphi }\left( x,\lambda \right) \right] } \\
	{=\widetilde{\varphi }\left( x,\lambda \right) S\left( x,\lambda \right)
		-\varphi \left( x,\lambda \right) \widetilde{S}\left( x,\lambda \right) +%
		\left[ M\left( \lambda \right) -\widetilde{M}\left( \lambda \right) \right]
		\varphi \left( x,\lambda \right) \widetilde{\varphi }\left( x,\lambda
		\right) } \\
	\\
	{P_{21}\left( x,\lambda \right) =\varphi ^{\prime }\left( x,\lambda \right) %
		\left[ \widetilde{S}^{{^{\prime }}}\left( x,\lambda \right) +\widetilde{M}%
		\left( \lambda \right) \cdot \widetilde{\varphi }^{{^{\prime }}}\left(
		x,\lambda \right) \right] -\widetilde{\varphi }^{{^{\prime }}}\left(
		x,\lambda \right) \left[ S^{\prime }\left( x,\lambda \right) +M\left(
		\lambda \right) \cdot \varphi ^{\prime }\left( x,\lambda \right) \right] }
	\\
	{=\varphi ^{\prime }\left( x,\lambda \right) \widetilde{S}^{{^{\prime }}%
		}\left( x,\lambda \right) -\widetilde{\varphi }^{{^{\prime }}}\left(
		x,\lambda \right) S^{\prime }\left( x,\lambda \right) +\left[ \widetilde{M}%
		\left( \lambda \right) -M\left( \lambda \right) \right] \varphi ^{\prime
		}\left( x,\lambda \right) \widetilde{\varphi }^{{^{\prime }}}\left(
		x,\lambda \right) } \\
	\\
	{P_{22}\left( x,\lambda \right) =\widetilde{\varphi }\left( x,\lambda
		\right) \left[ S^{\prime }\left( x,\lambda \right) +M\left( \lambda \right)
		\cdot \varphi ^{\prime }\left( x,\lambda \right) \right] +\varphi ^{\prime
		}\left( x,\lambda \right) \left[ \widetilde{S}\left( x,\lambda \right) +%
		\widetilde{M}\left( \lambda \right) \cdot \widetilde{\varphi }\left(
		x,\lambda \right) \right] } \\
	{=\varphi ^{\prime }\left( x,\lambda \right) S^{{^{\prime }}}\left(
		x,\lambda \right) -\varphi ^{{^{\prime }}}\left( x,\lambda \right)
		\widetilde{S}\left( x,\lambda \right) +\left[ M\left( \lambda \right) -%
		\widetilde{M}\left( \lambda \right) \right] \varphi ^{\prime }\left(
		x,\lambda \right) \widetilde{\varphi }\left( x,\lambda \right) }%
	\end{array}%
	$$
	from $M\left( \lambda \right) \equiv \widetilde{M}\left( \lambda \right) $;
	$$
	\begin{array}{l}
	{P_{11}\left( x,\lambda \right) =\varphi \left( x,\lambda \right) \widetilde{%
			S}^{{^{\prime }}}\left( x,\lambda \right) -\widetilde{\varphi }^{{^{\prime }}%
		}\left( x,\lambda \right) S\left( x,\lambda \right) } \\
	\\
	{P_{12}\left( x,\lambda \right) =\widetilde{\varphi }\left( x,\lambda
		\right) S\left( x,\lambda \right) -\varphi \left( x,\lambda \right)
		\widetilde{S}\left( x,\lambda \right) } \\
	\\
	{P_{21}\left( x,\lambda \right) =\varphi ^{\prime }\left( x,\lambda \right)
		\widetilde{S}^{{^{\prime }}}\left( x,\lambda \right) -\widetilde{\varphi }^{{%
				^{\prime }}}\left( x,\lambda \right) S^{\prime }\left( x,\lambda \right) }
	\\
	\\
	{P_{22}\left( x,\lambda \right) =\varphi ^{\prime }\left( x,\lambda \right)
		S^{{^{\prime }}}\left( x,\lambda \right) -\varphi ^{{^{\prime }}}\left(
		x,\lambda \right) \widetilde{S}\left( x,\lambda \right) }%
	\end{array}%
	$$
	are obtained. When $M\left( \lambda \right) \equiv \widetilde{M}\left( \lambda
	\right) $, it is clear that the $P_{j,k}\left( x,\lambda \right) ,\left(
	j,k=1,2\right) $ functions are full functions according to $\lambda $. From $%
	\left( 3.3\right) $; for $\forall x\in \left[ 0,\pi \right] $ , $c_{\delta }$
	, $C_{\delta }$ constants that provide $\left\vert P_{11}\left( x,\lambda
	\right) \right\vert \leq c_{\delta }$ and $\left\vert P_{12}\left( x,\lambda
	\right) \right\vert \leq C_{\delta }$ inequalities can be shown. From the
	Liouville theorem $P_{11}\left( x,\lambda \right) \equiv A\left( x\right) $
	and $P_{12}\left( x,\lambda \right) \equiv 0$. From
	$$
	\begin{array}{l}
	{\varphi \left( x,\lambda \right) \cdot \Phi ^{\prime }\left( x,\lambda
		\right) -\widetilde{\varphi }^{{^{\prime }}}\left( x,\lambda \right) \cdot
		\Phi \left( x,\lambda \right) =A\left( x\right) } \\
	{\widetilde{\varphi }\left( x,\lambda \right) \cdot \Phi \left( x,\lambda
		\right) -\varphi \left( x,\lambda \right) \cdot \widetilde{\Phi }\left(
		x,\lambda \right) =0}%
	\end{array}%
	$$
	\begin{equation}
	\left.
	\begin{array}{l}
	{\varphi \left( x,\lambda \right) =\widetilde{\varphi }\left( x,\lambda
		\right) \cdot A\left( x\right) } \\
	{\Phi \left( x,\lambda \right) =\widetilde{\Phi }\left( x,\lambda \right)
		\cdot A\left( x\right) }%
	\end{array}%
	\right\}   \label{GrindEQ__34_}
	\end{equation}
	
	are obtained and
	$$
	\begin{array}{l}
	{W\left[ \varphi ,\Phi \right] =W\left[ \varphi \left( x,\lambda \right) ,-%
		\frac{\psi \left( x,\lambda \right) }{\Delta \left( \lambda \right) }\right]
	} \\
	{=\frac{1}{\Delta \left( \lambda \right) }W\left[ \varphi \left( x,\lambda
		\right) ,-\psi \left( 0,\lambda \right) \varphi \left( x,\lambda \right)
		+\Delta \left( \lambda \right) S\left( x,\lambda \right) \right] } \\
	{=-\frac{\psi \left( 0,\lambda \right) }{\Delta \left( \lambda \right) }W%
		\left[ \varphi \left( x,\lambda \right) ,\varphi \left( x,\lambda \right) %
		\right] +W\left[ \varphi \left( x,\lambda \right) ,S\left( x,\lambda \right) %
		\right] =1}%
	\end{array}%
	$$
	And similarly  $W\left[\widetilde{\varphi },\widetilde{\Phi }\right]=1$ is obtained. If this equation is written in place of $\left(4.1\right)$,
	
	\[\begin{array}{l} {1=W\left[\varphi \left(x,\lambda \right),\Phi \left(x,\lambda \right)\right]=W\left[A\left(x\right)\widetilde{\varphi }\left(x,\lambda \right),A\left(x\right)\widetilde{\Phi }\left(x,\lambda \right)\right]} \\ {\, \, \, \, \, \, \, \, \, \, \, \, \, \, \, \, \, \, \, \, \, \, \, \, \, \, \, \, \, \, \, \, \, \, \, \, \, \, \, \, \, \, \, \, \, \, \, \, \, \, \, \, \, \, \, \, \, \, \, =A^{2} \left(x\right)W\left[\widetilde{\varphi }\left(x,\lambda \right),\widetilde{\Phi }\left(x,\lambda \right)\right]=A^{2} \left(x\right)} \end{array}\]
	
	is obtained.
	
	\noindent Therefore, $\left(\beta _{2}^{+} +\frac{\gamma _{2} }{2\beta }  \right)\ne 1$; $p_{1} =\tilde{p}_{1} \, \, ,\, \, p_{2} =\tilde{p}_{2} $ . We have $A\left(x\right)=1$ from  $\left(4.1\right)$
	
	\noindent $\varphi \left(x,\lambda \right)\equiv \tilde{\varphi }\left(x,\lambda \right)$ and $\Phi \left(x,\lambda \right)\equiv \tilde{\Phi }\left(x,\lambda \right)$.
	
	\noindent When $\varphi \left(x,\lambda \right)\equiv \tilde{\varphi }\left(x,\lambda \right)$,
	\[\begin{array}{l} {-\varphi ''+\left[2\lambda p\left(x\right)+q\left(x\right)\right]\varphi =\lambda ^{2} \delta \left(x\right)\varphi } \\ {-\varphi ''+\left[2\lambda p\left(x\right)+q\left(x\right)\right]\varphi =\lambda ^{2} \tilde{\delta }\left(x\right)\varphi } \end{array}\]
	
	are obtained.
	\[\left\{\lambda ^{2} \left(\delta \left(x\right)-\tilde{\delta }\left(x\right)\right)+2\lambda \left(p\left(x\right)-\tilde{p}\left(x\right)\right)+\left(q\left(x\right)-\tilde{q}\left(x\right)\right)\right\}\varphi \equiv 0 \left(for\, \, \forall \lambda \, \right)\]
	$\delta \left(x\right)=\tilde{\delta }\left(x\right)$, $p\left(x\right)=\tilde{p}\left(x\right)$ and $q\left(x\right)=\tilde{q}\left(x\right)$ a.e. For every $\lambda $ in discontinuity conditions,
	$$
	\begin{array}{l}
	{\left( \alpha _{1}-\tilde{\alpha}_{1}\right) \varphi \left( p_{1}-0,\lambda
		\right) =0} \\
	{\left( \beta _{1}-\tilde{\beta}_{1}\right) \varphi ^{\prime }\left(
		p_{1}-0,\lambda \right) +\left( \gamma _{1}-\tilde{\gamma}_{1}\right)
		\varphi \left( p_{1}-0,\lambda \right) =0}%
	\end{array}%
	$$
	$$
	\begin{array}{l}
	{\left( \alpha _{2}-\tilde{\alpha}_{2}\right) \varphi \left( p_{2}-0,\lambda
		\right) =0} \\
	{\left( \beta _{2}-\tilde{\beta}_{2}\right) \varphi ^{\prime }\left(
		p_{2}-0,\lambda \right) +\left( \gamma _{2}-\tilde{\gamma}_{2}\right)
		\varphi \left( p_{2}-0,\lambda \right) =0}%
	\end{array}%
	$$
	$\alpha _{1}=\tilde{\alpha}_{1},\,\,\beta _{1}=\tilde{\beta}_{1},\,\,\gamma
	_{1}=\tilde{\gamma}_{1}$ ve $\alpha _{2}=\tilde{\alpha}_{2},\,\,\beta _{2}=%
	\tilde{\beta}_{2},\,\,\gamma _{2}=\tilde{\gamma}_{2}$.\\ Consequently $L=%
	\tilde{L}$. The proof is completed.
\end{proof}

\section*{Acknowledgement} 

Not applicable.
\bibliography{mmnsample}
\bibliographystyle{mmn}

\end{document}